\documentclass[11pt]{article}
\usepackage{amsmath, amscd, amssymb, latexsym, epsfig, color, amsthm,enumerate, mathtools}
\usepackage[all]{xy}
\usepackage{verbatim}
\numberwithin{equation}{section}

\newtheorem{theorem}{Theorem}[section]

\newtheorem{corollary}[theorem]{Corollary}
\newtheorem{conjecture}[theorem]{Conjecture}
\newtheorem{lemma}[theorem]{Lemma}

\theoremstyle{definition}

\DeclareMathOperator\lk{\mathrm{lk}}

\DeclareMathOperator\st{\mathrm{st}}

\DeclareMathOperator{\conv}{\mathrm{conv}}
\DeclareMathOperator{\supp}{\mathrm{supp}}

\newcommand{\field}{{\bf k}}
\newcommand{\N}{{\mathbb N}}
\newcommand{\R}{{\mathbb R}}

\newcommand{\Z}{{\mathbb Z}}

\newcommand{\C}{{\mathcal C}}

\newcommand{\Stress}{\mathcal S}

\title{The stresses on centrally symmetric complexes and the lower bound theorems}
\author{
	Isabella Novik\thanks{Research of IN is partially\textsl{}
		supported by NSF grants DMS-1664865 and DMS-1953815, and by Robert R.~\&  Elaine F.~Phelps Professorship in Mathematics. }\\
	\small Department of Mathematics\\[-0.8ex]
	\small University of Washington\\[-0.8ex]
	\small Seattle, WA 98195-4350, USA\\[-0.8ex]
	\small \texttt{novik@uw.edu}
	\and 
	Hailun Zheng\thanks{Research of HZ is partially supported by a postdoctoral fellowship from ERC grant 716424 - CASe.} \\
	\small Department of Mathematical Sciences\\[-0.8ex]
	\small University of Copenhagen\\[-0.8ex]
	\small Universitesparken 5, 2100 Copenhagen, Denmark \\[-0.8ex]
	\small \texttt{hz@math.ku.dk}
}
\begin{document}
\maketitle
\begin{abstract}
In 1987, Stanley conjectured that if a centrally symmetric Cohen--Macaulay simplicial complex $\Delta$ of dimension $d-1$ satisfies $h_i(\Delta)=\binom{d}{i}$ for some $i\geq 1$, then $h_j(\Delta)=\binom{d}{j}$ for all $j\geq i$. Much more recently, Klee, Nevo, Novik, and Zheng conjectured that if a centrally symmetric simplicial polytope $P$ of dimension $d$ satisfies $g_i(\partial P)=\binom{d}{i}-\binom{d}{i-1}$ for some $d/2\geq i\geq 1$, then  $g_j(\partial P)=\binom{d}{j}-\binom{d}{j-1}$ for all $d/2\geq j\geq i$.
This note uses stress spaces to prove both of these conjectures.  
\end{abstract}
\section{Introduction}
This paper is devoted to analyzing the cases of equality in Stanley's lower bound theorems on the face numbers of centrally symmetric Cohen--Macaulay complexes and centrally symmetric polytopes. All complexes considered in this paper are simplicial.

In the seventies, Stanley and Hochster (independently from each other) introduced the notion of Stanley--Reisner rings and started developing their theory, see \cite{Hochster, Reisner, Stanley75, Stanley77}. In the fifty years since, this theory has become a major tool in the study of face numbers of simplicial complexes that resulted in a myriad of theorems and applications. Among them are a complete characterization of face numbers of Cohen--Macaulay (CM, for short) simplicial complexes \cite{Stanley77}, a complete characterization of flag face numbers of balanced CM complexes \cite{BjorFranklSta,Stanley79}, and a complete characterization of face numbers of simplicial polytopes \cite{BilleraLee, Stanley80}, to name just a few.

A simplicial complex $\Delta$ is called centrally symmetric (or cs) if its vertex set $V$ is endowed with a free involution $\alpha: V \to  V$ that induces a free involution on the set of all non-empty faces of $\Delta$. Motivated by the desire to understand face numbers of cs simplicial polytopes as well as to find a complete characterization of face numbers of cs CM complexes, Stanley \cite[Theorems 3.1 and 4.1]{Stanley-87} proved the following Lower Bound Theorem:
\begin{theorem} \label{Stanley-LBT-v1}
	Let $\Delta$ be a $(d-1)$-dimensional cs CM simplicial complex. Then $h_i(\Delta)\geq \binom{d}{i}$ for all $1\leq i\leq d$. Furthermore, if $\Delta$ is the boundary complex of a $d$-dimensional cs simplicial polytope, then $g_i(\Delta)\geq \binom{d}{i}-\binom{d}{i-1}$ for all $1\leq i\leq d/2$.
\end{theorem}
\noindent These inequalities are sharp: indeed, the boundary complex of the $d$-cross-polytope has $h_i=\binom{d}{i}$ for all $i$ and $g_i=\binom{d}{i}-\binom{d}{i-1}$ for all $1\leq i\leq d/2$. Stanley also proposed the following conjecture \cite[Conjecture 3.5]{Stanley-87}, which he verified in the case that $j$ is even or $j-i$ is even:
\begin{conjecture}\label{conj}
	Let $\Delta$ be a $(d-1)$-dimensional cs CM  simplicial complex. Suppose $h_i(\Delta)=\binom{d}{i}$ for some $i\geq 1$. Then $h_j(\Delta)=\binom{d}{j}$ for all $j\geq i$.
\end{conjecture}

Much more recently, Klee, Nevo, Novik, and Zheng \cite[Conjecture 8.5]{KNNZ} posited a conjecture that is similar in spirit, which they verified for $i=2$ (the case of $i=1$ is very easy):
 \begin{conjecture}\label{conj2}
	Let $\Delta$ be the boundary complex of a $d$-dimensional cs simplicial polytope. Suppose $g_i(\Delta)=\binom{d}{i}-\binom{d}{i-1}$ for some $d/2\geq i\geq 1$. Then $g_j(\Delta)=\binom{d}{j}-\binom{d}{j-1}$ for all $d/2\geq j\geq i$.
\end{conjecture}

In this note we prove both conjectures in full generality. The proofs are given in Section 3. 
Along the way, we show that any complex $\Delta$ satisfying conditions of Conjecture \ref{conj} contains the boundary complex of a $d$-cross-polytope as a subcomplex --- a fact that might be of independent interest. Our proof utilizes the theory of stress spaces developed by Lee \cite{Lee96}. Specifically, the $h$-numbers of a Cohen--Macaulay complex $\Delta$ can be viewed as the dimensions of certain spaces of linear stresses on $\Delta$ while the $g$-numbers of the boundary complex of a simplicial polytope are the dimensions of spaces of affine stresses. A key observation is that if $\Delta$ is a $(d-1)$-dimensional {\em cs} CM complex, then $h_i(\Delta)=\binom{d}{i}$ if and only if all linear $i$-stresses on $\Delta$ are symmetric; similarly, if $\Delta$ is the boundary complex of a $d$-dimensional {\em cs} simplicial polytope, then $g_i(\Delta)=\binom{d}{i}-\binom{d}{i-1}$ if and only if all affine $i$-stresses on $\Delta$ are symmetric, see the discussion in Section 2. Both conjectures then follow from the main result of the paper asserting that for an arbitrary cs simplicial complex $\Delta$, if $\Theta$ is a set of linear forms satisfying certain conditions and if for some $i>1$, all $i$-stresses on $\Delta$ computed w.r.t.~$\Theta$ are symmetric, then so are all $j$-stresses on $\Delta$ for any $j\geq i$, see Theorem \ref{main theorem prep}.

\section{Setting the stage}
We review several definitions and results on simplicial complexes, Stanley--Reisner rings, stress spaces, and Cohen--Macaulayness, as well as prepare ground for the proofs.  For all undefined terminology we refer the reader to \cite{Lee96, Stanley96}.

A(n abstract)  \emph{simplicial complex} $\Delta$ on the ground set $V$ is a collection of subsets of $V$ that is closed under inclusion; $v$ is a {\em vertex} of $\Delta$ if $\{v\}\in \Delta$, but not all elements of V are required to be vertices. The elements of $\Delta$ are called {\em faces}. The \emph{dimension of a face} $\tau\in\Delta$ is $\dim\tau:=|\tau|-1$. The \emph{dimension of $\Delta$}, $\dim\Delta$, is the maximum dimension of its faces. A face of a simplicial complex $\Delta$ is a \textit{facet} if it is maximal w.r.t.~inclusion. We say that $\Delta$ is \emph{pure} if all facets of $\Delta$ have the same dimension. To simplify notation, for a face that is a vertex, we write $v$ instead of $\{v\}$; we also define the following two subcomplexes of $\Delta$ called the {\em star of $v$} and the {\em link of $v$ in $\Delta$}:
$\st_\Delta(v)=\st(v):=\{\sigma\in\Delta \ : \  \sigma\cup v\in\Delta\}$ and $\lk_\Delta(v)=\lk(v):= \{\sigma\in \st_\Delta(v) \ : \ v\notin\sigma\}$.

Let $\Delta$ be a $(d-1)$-dimensional simplicial complex. For $-1 \leq i \leq d-1$, the $i$-th \textit{$f$-number} of $\Delta$, $f_i = f_i(\Delta)$, denotes the number of $i$-dimensional faces of $\Delta$. The \textit{$h$-numbers} of $\Delta$, $h_i=h_i(\Delta)$ for $0 \leq i \leq d$, are defined by the relation $\sum_{i=0}^{d}h_i\lambda^{d-i}=\sum_{i=0}^{d}f_{i-1}(\lambda-1)^{d-i}$. Finally, the \textit{$g$-numbers} of $\Delta$ are $g_0(\Delta):=1$ and $g_i(\Delta):=h_i(\Delta)-h_{i-1}(\Delta)$ for $1 \leq i \leq d/2$.

Let $\Delta$ be a simplicial complex on the ground set $V$. Let $X = \{x_v : v \in V\}$ be the set of variables and let $\R[X]$ be the polynomial ring over the real numbers $\R$ in variables $X$. The \emph{Stanley--Reisner ideal} of $\Delta$ is defined as \[I_\Delta=\left(x_{v_1}x_{v_2}\dots x_{v_i}: \ \{v_1, v_2, \dots, v_i\}\notin \Delta\right),\] i.e.,
it is the ideal generated by the squarefree monomials corresponding to non-faces of $\Delta$. The \emph{Stanley--Reisner ring} of $\Delta$ is $\R[\Delta] := \R[X]/I_\Delta$. The ring $\R[\Delta]$ has an $\N$-grading: $\R[\Delta]= \bigoplus_{i=0}^{\infty} \R[\Delta]_i$, where the $i$th graded component $\R[\Delta]_i$ is the space of homogeneous elements of degree $i$ in $\R[\Delta]$. In general, for an $\N$-graded vector space $M$, denote by $M_i$ the $i$th graded component of~$M$.

Let $\Delta$ be a simplicial complex and let $\Theta=\theta_1,\ldots,\theta_\ell$ be a sequence of linear forms in $\R[X]$, where $\ell$ is a nonnegative integer. Denote the quotient $\R[\Delta]/\Theta\R[\Delta]$ by $\R(\Delta,\Theta)$. 

For our proofs, we will work in the dual setting of stress spaces developed by Lee \cite{Lee96}, see also \cite[Section 3]{Adiprasito-g-conj}. It should also be mentioned that stress spaces are essentially the same objects  as {\em inverse systems}  in commutative algebra --- the notion that goes back to Macaulay; see \cite[Theorem 21.6 and Exercise  21.7]{Eisenbud}. Observe that a variable $x_v$ acts on $\R[X]$ by $\frac{\partial}{\partial{x_v}}$; for brevity, we will denote this operator by $\partial_{x_v}$. More generally, if $c(X)=\sum_{v\in V} c_v x_v$ is a linear form in $\R[X]$, then we define \begin{equation*}
\begin{split}
\partial_{c(X)} : \R[X]&\to\R[X], \\
w&\mapsto \sum_{v\in V}c_v\cdot\partial_{x_v}w=\sum_{v\in V}c_v\frac{\partial w}{\partial{x_v}}.
\end{split}
\end{equation*}
	
For a monomial $\mu\in \R[X]$, the {\em support} of $\mu$ is $\supp(\mu)=\{v\in V : x_v\,|\,\mu\}$. A homogeneous polynomial $w\in\R[X]$ of degree $i$ is called an {\em $i$-stress} on $\Delta$ w.r.t.~$\Theta=\theta_1,\ldots,\theta_\ell$ if it satisfies the following conditions:
	\begin{itemize}
	\item Every term $\mu$ of $w$ is supported on a face of $\Delta$: $\supp(\mu)\in\Delta$, and
	\item $\partial_{\theta_k}w=0$ for all $k=1,\ldots,\ell$.
	\end{itemize}
The {\em support} of an $i$-stress $w$, $\supp (w)$, is the subcomplex of $\Delta$ generated by the support of all terms of $w$. We say that a face $F\in\Delta$ {\em participates in a stress $w$} if $F\in \supp(w)$. We also say that a stress $w$ {\em lives on a subcomplex $\Gamma$} of $\Delta$ if $\supp(w)\subseteq \Gamma$.

Denote the set of all $i$-stresses on $\Delta$ w.r.t.~$\Theta$ by $\Stress(\Delta, \Theta)_i$. 
This set is a vector space \cite{Adiprasito-g-conj,Lee96}; it is a subspace of $\R[X]$. In fact, $\Stress(\Delta, \Theta)_i$ is the orthogonal complement of $(I_\Delta+(\Theta))_i$ in $\R[X]_i$ w.r.t.~a certain inner product on $\R[X]_i$, see \cite[Section 3]{Lee96}. Thus, as a vector space, $\Stress(\Delta, \Theta)_i$ is canonically isomorphic to $\R(\Delta, \Theta)_i$.  (For an alternative approach using the Weil duality, see \cite[Section 3]{Adiprasito-g-conj}.) Another very useful and easy fact is that for every linear form $c(X)\in \R[X]$, the operator $\partial_{c(X)}$ maps  $\Stress(\Delta,\Theta)_i$ into $\Stress(\Delta, \Theta)_{i-1}$, that is, if $w$ is a stress, then so is $\partial_{c(X)}w$. This follows from the fact that $\partial_{\theta_k}$ and $\partial_{c(X)}$ commute, and that a subset of a face of $\Delta$ is a face of $\Delta$.

Stresses are convenient to work with for the following reason: if $\Gamma$ is a subcomplex of $\Delta$ (considered as a complex on the same ground set $V$ as $\Delta$), then there is a natural surjective homomorphism $\rho:\R[\Delta] \to \R[\Gamma]$; it induces a surjective homomorphism $\R(\Delta,\Theta) \to \R(\Gamma,\Theta)$. On the level of stress spaces, the situation is much easier to describe: $\Stress(\Gamma,\Theta)_i$ is a subspace of  $\Stress(\Delta,\Theta)_i$.

A simplicial complex $\Delta$ is \textit{centrally symmetric} or {\em cs} if its ground set is endowed with a {\em free involution} $\alpha: V \rightarrow V$ that induces a free involution on the set of all non-empty faces of $\Delta$. In more detail,  for all non-empty faces $\tau \in \Delta$, the following holds: $\alpha(\tau)\in\Delta$, $\alpha(\tau)\neq \tau$, and $\alpha(\alpha(\tau))=\tau$. To simplify notation, we write $\alpha(\tau)=-\tau$ and refer to $\tau$ and $-\tau$ as {\em antipodal faces} of $\Delta$. 

A large family of cs simplicial complexes is given by cs simplicial polytopes. A \emph{polytope} $P\subset \R^d$ is the convex hull of a set of finitely many points in $\R^d$. We will always assume that $P$ is $d$-dimensional. A \emph{proper face} of $P$ is the intersection of $P$ with a supporting hyperplane. A polytope $P$ is called {\em simplicial} if all of its proper faces are geometric simplices, i.e., convex hulls of affinely independent points. We identify each face of a simplicial polytope $P$ with the set of its vertices. The {\em boundary complex} of $P$, denoted $\partial P$, is then the simplicial complex consisting of the empty set along with the vertex sets of proper faces of $P$. A polytope $P$ is called {\em cs} if $P=-P$; in this case, the complex $\partial P$ is a cs simplicial complex w.r.t.~the natural involution. An important example is $\partial\C^*_d$ --- the boundary complex of a $d$-cross-polytope $\C^*_d:=\conv(\pm p_1,\pm p_2,\ldots, \pm p_d)$, where $p_1,\ldots,p_d$ are affinely independent points in $\R^d\backslash\{0\}$. As an abstract simplicial complex, $\partial\C^*_d$ is the $d$-fold suspension of $\{\emptyset\}$. It is easy to check that $h_j(\partial \C_d^*)=\binom{d}{j}$ for all $0\leq j\leq d$, and so $g_j(\partial \C_d^*)=\binom{d}{j}-\binom{d}{j-1}$ for all $1\leq j\leq d/2$.

The free involution $\alpha$ on a cs complex $\Delta$ induces the free involution on $X$ via $\alpha(x_v)=x_{-v}$, which in turn induces a $\Z/2\Z$-action on $\R[X]$ and $\R[\Delta]$. For any $\R$-vector space $W$ endowed with such an action $\alpha$, one has $W=W^+\oplus W^-$, where $W^+:=\{w\in W: w=\alpha(w)\}$ and $W^-:=\{w\in W: w=-\alpha(w)\}$. Thus,
$\R[\Delta]_i=\R[\Delta]^+_i\oplus \R[\Delta]^-_i$. As $\R[\Delta]^+_{i}\cdot\R[\Delta]^-_j \subseteq \R[\Delta]^{-}_{i+j}$, and similar inclusions hold for all choices of plus and minus signs, it follows that $\R[\Delta]$ has an $(\N\times \Z/2\Z)$-grading. 

Let $\Delta$ be a cs simplicial complex with an involution $\alpha$, and let $\Theta=\theta_1,\ldots,\theta_\ell$ consist of linear forms that are homogeneous w.r.t.~the $(\N\times \Z/2\Z)$-grading. Since $\alpha(I_\Delta+(\Theta))=I_\Delta+(\Theta)$ and since for any $w,w'\in \R[X]_i$, $\langle\alpha(w),\alpha(w')\rangle=\langle w,w'\rangle$, where $\langle-,-\rangle$ is the inner product from \cite[Section 3]{Lee96} used to define the isomorphism $\Phi_i$ between $\R(\Delta, \Theta)_i$ and $\Stress(\Delta, \Theta)_i$, it follows that $\alpha$ also acts on $\Stress(\Delta, \Theta)_i$ and that this action commutes with $\Phi_i$. Hence, $\Stress(\Delta, \Theta)_i=\Stress(\Delta, \Theta)_i^+ \oplus \Stress(\Delta, \Theta)_i^-$, where the subspaces $\Stress(\Delta, \Theta)^+_i$ and $\Stress(\Delta, \Theta)^-_i$ of $\Stress(\Delta,\Theta)_i$ are isomorphic (as vector spaces) to $\R(\Delta, \Theta)^+_i$ and $\R(\Delta, \Theta)^-_i$, resp. We refer to the elements of $\Stress(\Delta, \Theta)^+_i$ as {\em symmetric $i$-stresses}. 

For certain classes of simplicial complexes and a certain choice of $\Theta$, the dimensions of stress spaces are well understood. This requires a few additional definitions. 
Let $\Delta$ be a $(d-1)$-dimensional simplicial complex. A sequence $\Theta=\theta_1,\ldots,\theta_\ell$ of linear forms in $\R[X]$ is called  a \emph{linear system of parameters of $\Delta$} (or \emph{l.s.o.p.}, for short) if $\ell=d$ and $\R(\Delta,\Theta)$ is a finite-dimensional $\R$-vector space. We say that $\Delta$ is \emph{Cohen--Macaulay} (or CM, for short) if for some (equivalently, every) l.s.o.p. $\Theta=\theta_1, \theta_2, \dots, \theta_d$ of $\Delta$, $$\dim_\R \R(\Delta,\Theta)_i=h_i(\Delta), \quad \forall \ 0\leq i\leq d.$$ In particular, if $\Delta$ is CM and $\Theta$ is an l.s.o.p.~of $\Delta$, then $\Stress(\Delta, \Theta)_i$ has dimension $h_i(\Delta)$. Following \cite{Lee96}, when $\Theta$ is an l.s.o.p.~of $\Delta$, we will refer to elements of $\Stress(\Delta, \Theta)_i$ as {\em linear} $i$-stresses.

It is worth mentioning that there are other equivalent definitions of CM complexes. The most standard one is that $\Delta$ is CM if some (equivalently, every) l.s.o.p.~of $\Delta$ is a regular sequence for the $\R[X]$-module $\R[\Delta]$. It is also worth mentioning that CM complexes have a topological characterization due to Reisner \cite{Reisner}. This characterization implies, for instance, that CM complexes are pure, that stars and links of CM complexes are also CM, and that the boundary complexes of simplicial polytopes are CM.\footnote{For any field $\field$, one may analogously define the rings $\field[\Delta]$ and $\field(\Delta,\Theta)$ as well as the notion of $\Delta$ being CM over $\field$. However, it follows from Reisner's criterion along with the universal coefficient theorem that if $\Delta$ is CM over some field $\field$, then $\Delta$ is CM over $\R$, i.e., $\Delta$ satisfies the definition given above. In other words, no generality is lost by working over $\R$.}

Stanley \cite{Stanley-87} showed that if $\Delta$ is a cs simplicial complex, then there exists an l.s.o.p.~$\Theta=\theta_1,\ldots,\theta_d$ of $\Delta$ with the property that each $\theta_k$ lies in $\R[X]^-_1$. We refer to such $\Theta$ as {\em Stanley's special l.s.o.p.} of $\Delta$; this object plays a crucial role in the proof of Conjecture \ref{conj}. In the case that $\Delta=\partial P$ is the boundary complex of a cs $d$-polytope $P\subset \R^d$, there is a canonical choice of Stanley's special l.s.o.p.~$\theta_1, \dots, \theta_d$ of $\Delta$ defined as follows: for $k=1,\ldots,d$,
\begin{equation} \label{canonical-lsop}
\theta_k=\sum_{v\in V} a_{v,k}x_v, \mbox{ where $a_{v,k}$ is the $k$-th coordinate of vertex $v\in P\subset \R^d$}.
\end{equation}
To prove Conjecture \ref{conj2} we will consider stresses on $\partial P$ w.r.t.~$\widetilde{\Theta}=\theta_1,\ldots,\theta_d,\theta_{d+1}$, where $\theta_1,\ldots,\theta_d$ are defined  by \eqref{canonical-lsop} and $\theta_{d+1}:=\sum_{v\in V} x_v$ is an element of $\R[X]^+_1$. We will refer to  $\widetilde{\Theta}$ as the {\em set of canonical linear forms} associated with $P$. Following \cite{Lee96},  the $i$-stresses on $\partial P$ w.r.t.~$\widetilde{\Theta}$ are called {\em affine $i$-stresses}.

The two main results of \cite{Stanley-87} (see proofs of Theorems 3.1 and 4.1 there) are the following Lower Bound Theorems for cs CM complexes and cs simplicial polytopes. 
\begin{theorem}\label{thm: cs CM LBT} 
Let $\Delta$ be a $(d-1)$-dimensional cs CM simplicial complex, and let $\Theta$ be Stanley's special l.s.o.p.~of $\Delta$. Then $$\dim_\R \R(\Delta, \Theta)^-_i=\frac{1}{2}\left(h_i(\Delta)-\binom{d}{i}\right) \quad \mbox{ for all $1\leq i\leq d$}.$$ In particular, $h_i(\Delta)\geq \binom{d}{i}$ for all $1\leq i\leq d$. 

Furthermore, if $\Delta=\partial P$ for some cs simplicial polytope $P$ and $\widetilde{\Theta}$ is the set of canonical linear forms associated with $P$, then $$\dim_\R \R(\Delta, \widetilde{\Theta})^-_i=\frac{1}{2}\left(g_i(\Delta)-\binom{d}{i}+\binom{d}{i-1}\right) \quad \mbox{ for all $1\leq i\leq d/2$}.$$ In particular, $g_i(\Delta)\geq \binom{d}{i}-\binom{d}{i-1}$ for all $1\leq i\leq d/2$. 
\end{theorem}

Using the language of stresses, Theorem \ref{thm: cs CM LBT} leads to the following:

\begin{corollary}\label{cor: minimal h_2 -> symmetry}
	Let $\Delta$ be a $(d-1)$-dimensional cs CM  simplicial complex, let $\Theta$ be Stanley's special l.s.o.p.~of $\Delta$, and let $1\leq i\leq d$ be an integer. Then  $h_{i}(\Delta)=\binom{d}{i}$ if and only if all linear $i$-stresses on $\Delta$ are symmetric, i.e., $\Stress(\Delta, \Theta)_{i}=\Stress(\Delta, \Theta)^+_{i}$. Furthermore, if $\Delta=\partial P$ for some cs simplicial polytope $P$, $\widetilde{\Theta}$ is the set of canonical linear forms associated with $P$, and $1\leq i\leq d/2$, then $g_i(\Delta)=\binom{d}{i}-\binom{d}{i-1}$ if and only if all affine $i$-stresses on $\Delta$ are symmetric, i.e., $\Stress(\Delta, \widetilde{\Theta})_{i}=\Stress(\Delta, \widetilde{\Theta})^+_{i}$.
\end{corollary}
\begin{proof}
Recall that $\R(\Delta,\Theta)^-_i \cong \Stress(\Delta,\Theta)^-_i$ and  $\R(\partial P, \widetilde{\Theta})^-_i \cong \Stress(\partial P, \widetilde{\Theta})^-_i$.	Theorem \ref{thm: cs CM LBT} then implies that $\Stress(\Delta, \Theta)^-_i=(0)$ if and only if $h_i(\Delta)=\binom{d}{i}$, and that $\Stress(\partial P,  \widetilde{\Theta})^-_i=(0)$ if and only if $g_i(\Delta)=\binom{d}{i}-\binom{d}{i-1}$.
\end{proof}


\section{Proof of the conjectures}
With the tools of Section 2 at our disposal, we are ready to prove Conjectures \ref{conj} and \ref{conj2}. In fact, we prove a more general result, Theorem \ref{main theorem prep}, from which the conjectures readily follow. To simplify notation, we assume that $V=\{\pm 1,\pm 2,\ldots,\pm n\}$ and let $[j]$ denote the set $\{1,2,\ldots,j\}$. We also refer to the elements of $\R[X]^+_i$ as {\em symmetric $i$-polynomials}.

We start with two simple lemmas.

\begin{lemma} \label{lm: support sttau}
	Let $\Delta$ be a cs simplicial complex and let $\Theta=\theta_1,\ldots,\theta_\ell$ be linear forms in $\R[X]$ that are homogeneous w.r.t.~the $(\N \times \Z/2\Z)$-grading. Let $v$ be a vertex of $\Delta$. If $w$ is a symmetric stress on $\Delta$ that lives on $\st(v)$, then, in fact, $w$ lives on $\lk(v)\cap\lk(-v)$.
\end{lemma}
\begin{proof} By the definition of cs complexes, $-v\notin \st(v)$. Thus the assumption that $w$ is symmetric and lives on $\st(v)$ implies that $w$ lives on $\lk(v)$. Now, since $w$ is symmetric, a face $F$ of $\Delta$ participates in $w$ if and only if $-F$ does. This together with the symmetry of $\Delta$  yields that $w$ lives on $\lk(v)\cap\lk(-v)$.
\end{proof}

\begin{lemma} \label{lm: squarefree}
Let $\Delta$ be a cs simplicial complex, let $\Theta=\theta_1,\ldots,\theta_\ell$ be linear forms in $\R[X]$ that are homogeneous w.r.t.~the $(\N \times \Z/2\Z)$-grading, and let $w\in \mathcal{S}(\Delta, \Theta)_i$. If for every vertex $v$, $\partial_{x_v} w$ is a symmetric stress, then $w$ is a squarefree polynomial.
\end{lemma}
\begin{proof}
 If $v$ is in the support of $w$, then $\partial_{x_v}w$ is a symmetric stress that lives on $\st(v)$. Hence by Lemma \ref{lm: support sttau}, $\partial_{x_v}w$ lives on $\lk(v)$. In particular, no term of $w$ is divisible by $x_v^2$.
\end{proof}

The following two lemmas provide key ingredients for the proof of Theorem \ref{main theorem prep}. For $k\in [n]$, we let $y_k$ denote $x_k+x_{-k}$.

\begin{lemma}\label{lm: symmetric stress expression}
Let $w \in \R[X]_i$ be a squarefree symmetric polynomial such that $\partial_{x_v} w$ is symmetric for all vertices $v$. Then $w$ is a squarefree polynomial in $y_1,\ldots,y_n$, that is, $w$ can be written as $$w=\sum_{\substack{\tau\subseteq [n]\\|\tau|=i}} c_\tau \prod_{k\in \tau} (x_k+x_{-k}) \quad \mbox{for some $c_\tau\in \R$}.$$ 
\end{lemma}
\begin{proof}
 It is easy to prove by induction on $n$ that a squarefree polynomial $Q\in\R[X]$ is a  polynomial in $y_1,\ldots,y_n$ if and only if $\partial_{x_{k}}Q=\partial_{x_{-k}}Q$ for all $k\in [n]$. Thus 
to prove  the lemma, it is enough to check that our given $w$ satisfies $\partial_{x_{k}}w=\partial_{x_{-k}}w$ for all $k\in [n]$. Indeed, by symmetry of $w$ and $\partial_{x_k} w$, and by the definition of $\alpha$,
\[ 
\partial_{x_{k}}w=\alpha(\partial_{x_{k}}w)=\partial_{x_{-k}} (\alpha w)=\partial_{x_{-k}} w.
\]
The result follows.
\end{proof}

\begin{lemma}\label{lm: i-stress to i+1-stress}
Let $i\geq 1$ and let $w\in \R[X]_{i+1}$ be a squarefree polynomial such that for all vertices $v$, $\partial_{x_v} w$ is a polynomial in $y_1,\ldots,y_n$. Then $w$ is a squarefree polynomial in $y_1,\ldots,y_n$. In particular, $w$ is symmetric and can be expressed as $$w=\sum_{\substack{\sigma\subseteq [n]\\|\sigma|=i+1}} c_\sigma \prod_{k\in \sigma} (x_k+x_{-k}) \quad \mbox{for some }c_\sigma\in\R.$$
\end{lemma}
\begin{proof}
By Lemma \ref{lm: symmetric stress expression}, the statement will follow if we show that $w$ is symmetric. To check this, write $w$ as $w=\sum c_{k_1,k_2,\ldots,k_{i+1}} x_{k_1}x_{k_2}\cdots x_{k_{i+1}}$ for some $c_{k_1,k_2,\ldots,k_{i+1}}\in \R$. The assumption that partial derivatives of $w$ are polynomials in $y_1,\ldots,y_n$ implies that $\partial_{x_{k_2}}\cdots \partial_{x_{k_{i+1}}} w$ is symmetric. Hence $c_{k_1,k_2,\ldots,k_{i+1}}= c_{-k_1,k_2,\ldots,k_{i+1}}$ (as they are coefficients of $x_{k_1}$ and $x_{-k_1}$ in $\partial_{x_{k_2}}\cdots \partial_{x_{k_{i+1}}} w$). Repeated applications of this argument imply that $c_{k_1,k_2,\ldots,k_{i+1}}= c_{-k_1,-k_2,\ldots,-k_{i+1}}$. Thus, $w$ is symmetric.
\end{proof}

We are now in a position to state and prove our main result.

\begin{theorem}\label{main theorem prep}
	Let $\Delta$ be a cs complex, and let $\Theta=\theta_1,\ldots,\theta_\ell$ be linear forms such that $\theta_1,\ldots,\theta_{\ell-1}$ are elements of $\R[X]^-_1$, and $\theta_\ell$ is either also in $\R[X]^-_1$ or $\theta_\ell=\sum_{v\in V} x_v$. If for some integer $i>1$, all $i$-stresses on $\Delta$ w.r.t.~$\Theta$ are symmetric, i.e., $\Stress(\Delta,\Theta)_i=\Stress(\Delta,\Theta)^+_i$, then for all $j\geq i$, $\Stress(\Delta,\Theta)_{j}=\Stress(\Delta,\Theta)^+_{j}$. Furthermore, if $\Stress(\Delta,\Theta)_{j}\neq (0)$ for some $j> i$, then $\Delta$ contains the boundary complex of the $j$-cross-polytope as a subcomplex.
\end{theorem}
\begin{proof}
	It suffices to prove the statement for $j=i+1$. Let $w\in \mathcal{S}(\Delta, \Theta)_{i+1}$. For every vertex $v$, $\partial_{x_v} w\in \mathcal{S}(\Delta, \Theta)_i$, and so $\partial_{x_v} w$ is symmetric. Hence, by Lemma \ref{lm: squarefree}, $w$ is squarefree. 
	
	Consider an edge $\{u_1, u_2\}\in \supp(w)$. Then $\partial_{x_{u_1}} w$ is a symmetric $i$-stress that lives on $\st( u_1)$, and so by Lemma \ref{lm: support sttau}, it lives on $\lk (u_1)\cap \lk (-u_1)$. Consequently, the stress $\partial_{x_{u_2}}\partial_{x_{u_1}} w$  lives on $\lk(u_1)\cap \lk(-u_1)$. Since $\partial_{x_{u_2}}\partial_{x_{u_1}} w=\partial_{x_{u_1}}\partial_{x_{u_2}} w$, the same argument implies that it also lives on $\lk(u_2)\cap \lk(-u_2)$. 
	Let $$w':=(x_{u_1}+x_{-u_1}-x_{u_2}-x_{-u_2})\cdot \partial_{x_{u_2}}\partial_{x_{u_1}} w.$$ Our discussion shows that $\supp(w')\subseteq \Delta$. Furthermore, by our assumptions on $\Theta$ and the fact that $w\in \mathcal{S}(\Delta, \Theta)_{i+1}$, it follows that $\partial_{\theta_k} w=0$ and $\partial_{\theta_k}(x_{u_1}+x_{-u_1}-x_{u_2}-x_{-u_2})=0$ for all $1\leq k\leq \ell$. Therefore, for all $1\leq k\leq \ell$,
	$$\partial_{\theta_k} w'=\partial_{\theta_k}(x_{u_1}+x_{-u_1}-x_{u_2}-x_{-u_2}) \cdot \partial_{x_{u_2}}\partial_{x_{u_1}} w+(x_{u_1}+x_{-u_1}-x_{u_2}-x_{-u_2}) \cdot\partial_{x_{u_2}}\partial_{x_{u_1}} \partial_{\theta_k}w=0.$$ Hence $w'\in \mathcal{S}(\Delta, \Theta)_i$, and so it is symmetric.  We conclude that $\partial_{x_{u_2}}\partial_{x_{u_1}} w\in \mathcal{S}(\Delta, \Theta)^+_{i-1}$ for any $u_2\in \supp(\partial_{x_{u_1}} w)$. Since the stress $\partial_{x_{u_1}}w$ itself is symmetric (indeed, it is an $i$-stress), Lemma \ref{lm: symmetric stress expression} guarantees that $\partial_{x_{u_1}}w$ is of the form $\partial_{x_{u_1}} w=\sum_{\tau\subseteq [n], \, |\tau|=i} c_\tau \prod_{k\in \tau} (x_k+x_{-k})$, for all $u_1\in\supp(w)$. It then follows from Lemma \ref{lm: i-stress to i+1-stress} that $w$ is a symmetric stress of the form $w=\sum_{\sigma\subseteq [n], \, |\sigma|=i+1} c_\sigma \prod_{k\in \sigma} (x_k+x_{-k})$. In particular, we see from the definition of stresses that if $w\neq 0$, then the support of $w$ is the union of the boundary complexes of $(i+1)$-cross-polytopes. This completes the proof.
\end{proof}

The proof of Conjectures \ref{conj} and \ref{conj2} now readily follows. In the proof, we use linear and affine stresses, i.e., stresses w.r.t.~Stanley's special l.s.o.p.~$\Theta$ and w.r.t.~the set of canonical linear forms $\widetilde{\Theta}$, respectively.

\begin{theorem}\label{main theorem} \qquad
\begin{enumerate}
\item 
Let $d$ and $1\leq i<d$ be integers. Let $\Delta$ be a cs CM complex of dimension $d-1$ with $h_i(\Delta)=\binom{d}{i}$. Then $h_j(\Delta)=\binom{d}{j}$ for all $i\leq j\leq d$. 

\item Let $d$ and $1\leq i<d/2$ be integers. If $\Delta=\partial P$ for some cs simplicial $d$-polytope $P$ and $g_i(\Delta)=\binom{d}{i}-\binom{d}{i-1}$, then $g_j(\Delta)=\binom{d}{j}-\binom{d}{j-1}$ for all $i\leq j\leq d/2$. 
\end{enumerate}
\end{theorem}
\begin{proof}
	We begin with the case of $i>1$. For the first part, let $\Theta$ be Stanley's special l.s.o.p.~of $\Delta$. Since $h_i(\Delta)=\binom{d}{i}$, it follows from Corollary \ref{cor: minimal h_2 -> symmetry} that all linear $i$-stresses on $\Delta$ are symmetric. By Theorem \ref{main theorem prep}, all linear $j$-stresses (for any $j\geq i$) are also symmetric. Hence Corollary \ref{cor: minimal h_2 -> symmetry} yields the result. The proof of the second part is analogous: this time use $\widetilde{\Theta}$ --- the set of canonical linear forms associated with $P$ --- and then apply Corollary \ref{cor: minimal h_2 -> symmetry} and Theorem \ref{main theorem prep} to affine stresses.
	
	Next we deal with the case of $i=1$ in both parts. The assumption that $h_1(\Delta)=d$, or that $g_1(\Delta)=d-1$, is equivalent to $f_0(\Delta)=2d$. Now, it follows easily from the definition of cs complexes that any cs complex on $2d$ vertices is contained in the boundary complex of the $d$-cross-polytope, and so $\Delta\subseteq \partial \C^*_d$. Since $\Delta$ and $\partial \C^*_d$ are CM complexes of the same dimension, \cite[Theorem 2.1]{Stanley-93} implies that $h_j(\Delta)\leq h_j(\partial \C^*_d)=\binom{d}{j}$ for all $j$. On the other hand, according to Theorem \ref{Stanley-LBT-v1}, $h_j(\Delta)\geq \binom{d}{j}$ for all $j$. Thus we must have  $h_j(\Delta)=\binom{d}{j}$ for all $j$, and hence also $g_j(\Delta)=\binom{d}{j}-\binom{d}{j-1}$ for all $j$. (Moreover, that the two complexes $\Delta\subseteq \partial\C^*_d$ have the same $h$-numbers yields that they have the same $f$-numbers, and so, in fact, $\Delta\cong\partial\C^*_d$.)
\end{proof}

It is worth remarking that under the conditions of Theorem \ref{main theorem}, we can say a bit more about~$\Delta$:

\begin{corollary} \qquad
\begin{enumerate}
\item Let $\Delta$ be a $(d-1)$-dimensional cs CM simplicial complex with $h_i(\Delta)=\binom{d}{i}$ for some $1\leq i <d$. Then $\Delta$ contains a subcomplex $\Gamma$ isomorphic to $\partial\C^*_d$. Furthermore, $\Stress(\Delta, \Theta)_j=\Stress(\Gamma, \Theta)_j$ for all $j\geq i$, where $\Theta$ is Stanley's special l.s.o.p.~of $\Delta$. 
\item Let $\Delta=\partial P$ where $P$ is a cs simplicial $d$-polytope. If $g_i(\Delta)=\binom{d}{i}-\binom{d}{i-1}$ for some $1\leq i \leq (d-2)/2$, then $\Delta$ contains $\partial\C^*_{\lfloor d/2\rfloor}$ as a subcomplex.  
\end{enumerate}
\end{corollary}
\begin{proof}
	If $i=1$, then the proof of Theorem \ref{main theorem} implies that in both parts $\Delta\cong\partial\C^*_d$. Thus assume that $i>1$. For the second statement, since by Theorem \ref{main theorem},	$g_{\lfloor d/2\rfloor}(\Delta)=\binom{d}{\lfloor d/2\rfloor}-\binom{d}{\lfloor d/2\rfloor-1}>0$, it follows that $\Stress(\Delta, \widetilde{\Theta})_{\lfloor d/2\rfloor}\neq (0)$, where $\widetilde{\Theta}$ is the set of canonical linear forms associated with $P$. Since by our assumptions, $\Stress(\Delta,\widetilde{\Theta})_i=\Stress(\Delta,\widetilde{\Theta})^+_i$ and $\lfloor d/2\rfloor >i$, Theorem \ref{main theorem prep} guarantees that $\Delta$ contains $\partial\C^*_{\lfloor d/2\rfloor}$ as a subcomplex.

The proof of the first statement is similar: since by Theorem \ref{main theorem}, $h_d(\Delta)=1$, there is a non-zero linear $d$-stress $w$ on $\Delta$. Since $d>i$ and $\Stress(\Delta,\Theta)_i=\Stress(\Delta,\Theta)^+_i$, Theorem \ref{main theorem prep} implies that $\Delta$  must contain $\Gamma\cong\partial\C^*_d$ as a subcomplex. Then $\Stress(\Delta, \Theta)_j\supseteq\Stress(\Gamma, \Theta)_j$ for all $j$, and comparing the dimensions we see that, in fact, $\Stress(\Delta, \Theta)_j=\Stress(\Gamma,\Theta)_j$ for all $j\geq i$.
\end{proof}

\subsection*{Acknowledgments} We are grateful to Satoshi Murai, Eran Nevo, Richard Stanley, and the two anonymous referees for comments on the previous versions of this note.
{\small
	\bibliography{refs}
	\bibliographystyle{plain}
}
\end{document}